\documentclass[12pt]{amsart}
\usepackage{amsfonts}
\usepackage{amssymb}
\usepackage{enumerate}
\usepackage{graphicx}

\makeatletter
\@namedef{subjclassname@2010}{  \textup{2010} Mathematics Subject Classification}
\makeatother
\newtheorem{theorem}{Theorem}[section]
\newtheorem{lemma}[theorem]{Lemma}

\newtheorem{quest}[theorem]{Question}

\newcommand{\thistheoremname}{}
\newtheorem*{genericthm*}{\thistheoremname}
\newenvironment{namedthm*}[1]
  {\renewcommand{\thistheoremname}{#1}
  \begin{genericthm*}}
  {\end{genericthm*}}
\theoremstyle{definition}
\newtheorem{definition}[theorem]{Definition}

\numberwithin{equation}{section}
\DeclareMathOperator{\diam}{diam}

\DeclareMathOperator{\co}{co}

\begin{document}
\title[Nonexpansive dynamical systems]{Amenable semigroups and nonexpansive
dynamical systems}
\author[A. Wi\'{s}nicki]{Andrzej Wi\'{s}nicki}
\address{Department of Mathematics, Pedagogical University of Krakow,
PL-30-084 Cracow, Poland}
\email{andrzej.wisnicki@up.krakow.pl}
\date{}

\begin{abstract}
We characterize amenability of subspaces of $C(S)$, where $S$ is a
semitopological semigroup, in terms of fixed point properties of
nonexpansive actions. In particular, we give a complete characterization of
a semitopological semigroup with a left invariant mean on WAP(S) that
answers a question of A.T.-M. Lau and Y. Zhang in the affirmative. We also
propose a new approach to Lau's problem concerning a counterpart of
Day-Mitchell's characterization of amenable semigroups and show some partial
results, in the case of weak$^{\ast }$ compact convex sets with the
Radon-Nikod\'{y}m property, and in the duals of $M$-embedded Banach spaces.
\end{abstract}

\maketitle



\section{Introduction}

A strong connection between amenability and fixed point properties of
semigroup actions was investigated in a number of papers. We recall only a
few classical results. Day \cite{Da} characterized amenable semigroups in
terms of the fixed point property for continuous affine mappings acting on
compact convex sets in a locally convex space and Mitchell \cite{Mi}
extended Day's result to semitopological semigroups. Ryll-Nardzewski \cite%
{Ry0} used his fixed point theorem to show the existence of the left
invariant mean on the space $WAP(G)$ of weakly almost periodic functions on
a group $G.$ Amenability of various subspaces of $C(S)$ are often
characterized in the context of affine actions on a semigroup $S$ (see,
e.g., \cite{La, La1, Mi, Pa}) and the question naturally arises as to
whether similar description can be given with the use of nonexpansive, i.e.,
$1$-Lipschitz actions. In 1973, Lau \cite{La} described amenability of $%
AP(S) $, the space of almost periodic functions on a semigroup $S$, in terms
of the fixed point property for nonexpansive mappings acting on compact
convex sets and it appears to be the only full characterization of this type.

This paper is motivated by the recent study of Lau and Y. Zhang on fixed
point properties of semigroups of nonexpansive mappings (see \cite{LaZh1}--%
\cite{LaZh4}). Several relations among fixed point properties on weakly and
weak$^{\ast }$ compact convex sets are discussed in these papers (see the
diagram on page 2553 in \cite{LaZh1}). In particular, \cite[Theorem 3.4]%
{LaZh1} gives a characterization of a separable semitopological semigroup $S$
with a left invariant mean on $WAP(S)$ in terms of the fixed point property
for nonexpansive mappings acting on weakly compact convex subsets of a
locally convex space and Question 4 in \cite{LaZh2} asks about the full such
characterization.

In this paper we take forward this program. In Section 3, we apply the
notion of fragmentability to drop the separability assumption from some
results in \cite{LaZh1}. In particular, we give a complete characterization
of a semitopological semigroup that has a left invariant mean on weakly
almost periodic functions and thus answer \cite[Question 4]{LaZh2} in the
affirmative (see Theorem \ref{WAP}). We also extend a fixed point theorem of
Hsu \cite{Hsu} (see also \cite[Theorem 3.10]{LaZh2}) from left reversible
discrete semigroups to left amenable semitopological semigroups. Note that
the case of Banach spaces was studied in \cite{Wi1} but the general case of
locally convex spaces is more complicated. Another open problem in \cite[p.
2542]{LaZh1} (see also \cite[Problem 1]{LaZh4}) concerns the chain of
implications: $(G)\Rightarrow (F)\Rightarrow (E)\Rightarrow (D)$ (see
Section 2 for the definitions). We show that $(D)\Leftrightarrow (E)$ (the
relation $(E)\not\Rightarrow (F)$ was proved in \cite{LaZh1}, and whether $%
(F)\Rightarrow (G)$ is unknown, see also the diagram below).

Another aspect of our work that has not yet been studied is the application
of the Bruck retraction method to get some qualitative results about the set
of fixed points of $S$ (see Theorem \ref{q}).

The following diagram summarizes the relations among the fixed point
properties of semitopological semigroups acting on weakly compact convex
sets in locally convex spaces discussed in Section 3 (compare the diagram on
p. 2553 in \cite{LaZh1}):

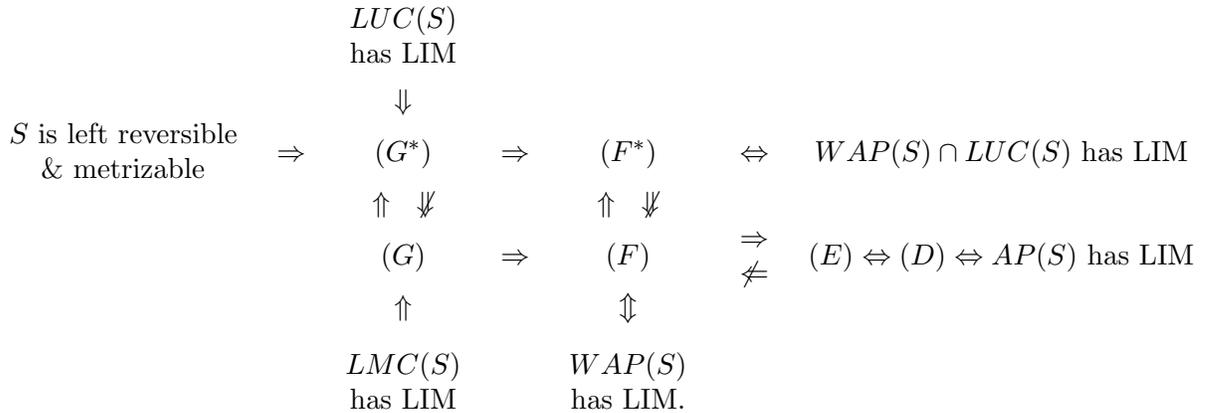
\bigskip \bigskip

\hbox{\hspace{-1.5cm}
{\small $\begin{array}{ccccccc}
&  &
\begin{array}{c}
LUC(S) \\
\text{has LIM}\end{array}\medskip &  &  &  &  \\
&  & \Downarrow &  &  &  &  \\
\begin{array}{c}
S\text{ is left reversible} \\
\&\text{ metrizable}\end{array}
& \Rightarrow & (G^{\ast }) & \Rightarrow & (F^{\ast }) & \Leftrightarrow &
WAP(S)\cap LUC(S)\text{ has LIM} \\
&  &
\begin{array}{cc}
\Uparrow & \not\Downarrow\end{array}
&  &
\begin{array}{cc}
\Uparrow & \not\Downarrow\end{array}
&  &  \\
&  & (G) & \Rightarrow & (F) &
\begin{array}{c}
\Rightarrow \\
\not\Leftarrow\end{array}
& (E)\Leftrightarrow (D)\Leftrightarrow AP(S)\text{ has LIM} \\
&  & \Uparrow \medskip &  & \Updownarrow \medskip &  &  \\
&  &
\begin{array}{c}
LMC(S) \\
\text{has LIM}\end{array}
&  &
\begin{array}{c}
WAP(S) \\
\text{has LIM.}\end{array}
&  &
\end{array}$}
}

\bigskip \bigskip \bigskip

An old problem in fixed point theory, posed by A. T.-M. Lau in 1976 (see
\cite[Problem 4]{La1}, \cite[Question 1]{LaZh2}), concerns a counterpart of
the well-known Day-Mitchell characterization of amenable semigroups. In
Section 4, we extend the techniques of Section 3 to propose a new approach
to Lau's problem by reducing it to a certain question about the norm support
of a weak$^{\ast }$ Borel Radon measure on a weak$^{\ast }$ compact set. As
a consequence, we show that Lau's problem has an affirmative solution in the
case of weak$^{\ast }$ compact convex sets with the Radon-Nikod\'{y}m
property, in particular for subsets of the dual of an Asplund space. We also
extend the corresponding results in \cite{LaZh2} concerning $L$-embedded
subsets of Banach spaces. In particular, Lau's problem is also solved in the
affirmative for weak$^{\ast }$ compact convex sets in the duals of $M$%
-embedded Banach spaces.

\section{Preliminaries}

Let $S$ be a semitopological semigroup, i.e., a semigroup with a Hausdorff
topology such that the mappings $S\ni s\rightarrow ts$ and $S\ni
s\rightarrow st$ are continuous for each $t\in S.$ Let $\ell ^{\infty }(S)$
be the Banach space of bounded complex-valued functions on $S$ with the
supremum norm. For $s\in S$ and $f\in \ell ^{\infty }(S),$ we define the
left and right translations of $f$ in $\ell ^{\infty }(S)$ by%
\begin{equation*}
L_{s}f(t)=f(st),\ R_{s}f(t)=f(st)
\end{equation*}%
for every $t\in S.$ Let $Y$ be a closed linear subspace of $\ell ^{\infty
}(S)$ containing constants and invariant under translations, i.e., $%
L_{s}(Y)\subset Y,$ $R_{s}(Y)\subset Y.$ Then a linear functional $\mu \in
Y^{\ast }$ is called a left (respectively right) invariant mean on $Y$ (LIM
or RIM, for short), if $\left\Vert \mu \right\Vert =\mu (1)=1$ and $\mu
(L_{s}f)=\mu (f)$ (respectively $\mu (R_{s}f)=\mu (f)$) for each $s\in S$
and $f\in Y.$

Let $C(S)$ be the closed subalgebra of $\ell ^{\infty }(S)$ consisting of
bounded continuous functions and let $f\in C(S)$. We will say that $f\in
LUC(S)$ $(f\in WLUC(S))$ if the mapping $S\ni s\rightarrow L_{s}f$ from $S$
to $C(S)$ is continuous when $C(S)$ has the norm topology (weak-topology).
Similarly, $f\in LMC(S)$ if the mapping $S\ni s\rightarrow L_{s}f\in C(S)$
is continuous when $C(S)$ is given the topology induced by the
multiplicative means on $C(S).$ (Recall that a mean $\mu $ is multiplicative
if $\mu (fg)=\mu (f)\mu (g)$). A bounded continuous function $f$ on $S$ is
called almost periodic (weakly almost periodic) if $\{L_{s}f:s\in S\}$ is
relatively compact in the norm topology (weak topology) of $C(S).$ The space
of almost periodic functions on $S$ is denoted by $AP(S)$ and the space of
weakly almost periodic functions by $WAP(S)$. In general,
\begin{eqnarray*}
AP(S) &\subset &LUC(S),\ AP(S)\subset WAP(S), \\
LUC(S) &\subset &WLUC(S)\subset LMC(S)\subset C(S),
\end{eqnarray*}%
and if $S$ is discrete, then%
\begin{equation*}
AP(S)\subset WAP(S)\subset LUC(S)=\ell ^{\infty }(S).
\end{equation*}%
It is not difficult to show that all these spaces are translation-invariant.
A semigroup $S$ is called left amenable if there exists a left invariant
mean on $LUC(S).$

Let $K$ be a topological space. A semigroup $S$ is said to act on $K$ (from
the left) if there is a map $\pi :S\times K\rightarrow K$ such that $%
s_{1}(s_{2}x)=(s_{1}s_{2})x$ for all $s_{1},s_{2}\in S$ and $x\in K$, where
as usual, we write $sx$ instead of $\pi (s,x)$. We say that the action $\pi $
is separately continuous if all orbit maps $\rho _{x}:S\rightarrow K,\rho
_{x}(s)=sx,$ and all translations $\lambda _{s}:K\rightarrow K,\lambda
_{s}(x)=sx,$ are continuous.

By a dynamical system we mean a pair $(S,K)$, where $S$ is a semitopological
semigroup, $K$ a topological space and there is a separately continuous
semigroup action $\pi :S\times K\rightarrow K$ of $S$ on $K.$ We will say
that a dynamical system is continuous if the action $\pi $ is jointly
continuous. A dynamical system is said to be compact if $K$ is compact. The
enveloping (or Ellis) semigroup $E(S,K)$ of a compact dynamical system $%
(S,K) $ is the closure in the product topology of the set $\{\lambda
_{s}:K\rightarrow K:s\in S\}$ of $S$-translations in the compact semigroup $%
K^{K}.$

Let $(X,\tau )$ be a locally convex space whose topology $\tau $ is
determined by a family $Q$ of (continuous) seminorms on $X$ and let $%
K\subset X$. We say that a dynamical system $(S,K)$ is $Q$-nonexpansive (or
briefly, nonexpansive if $Q$ is fixed) if $p(sx-sy)\leq p(x-y)$ for every $%
p\in Q,s\in S$ and $x,y\in K.$ A subset $K$ of $X$ is said to have $Q$%
-normal structure if for each $Q$-bounded subset $A$ of $K$ that contains
more than one point, there is $u\in \overline{%
\co%
}K$ and $p\in Q$ such that $\sup \{p(x-u):x\in A\}<\sup \{p(x-y):x,y\in A\}.$
Recall that $A$ is $Q$-bounded if for each $p\in Q$ there is $d>0$ such that
$p(x)\leq d$ for all $x\in A.$ Finally, a system $(S,K)$ is said to have a
fixed point if there is $x\in K$ such that $sx=x$ for every $s\in S.$

Let $S$ be a semitopological semigroup and $(X,\tau )$ a locally convex
space whose topology is determined by a family $Q$ of seminorms on $X$.
Following \cite{LaZh1}, we shall consider the following fixed point
properties for a semigroup $S$: 

\begin{enumerate}
\item[$(D)$:] \textit{Every nonexpansive dynamical system }$(S,(K,\tau ))$%
\textit{, where }$K$\textit{\ is a compact convex subset of a locally convex
space, has a fixed point.}

\item[$(E)$:] \textit{Every nonexpansive system }$(S,(K,\mathrm{weak}))$%
\textit{, where }$K$\textit{\ is a weakly compact convex subset of a locally
convex space, such that the translations }$\lambda _{s}:K\rightarrow K$%
\textit{\ are equicontinuous (in weak topology), has a fixed point.}

\item[$(E^{\prime })$:] \textit{Every nonexpansive system }$(S,(K,\mathrm{%
weak}))$\textit{, where }$K$\textit{\ is weakly compact convex and has }$Q$%
\textit{-normal structure, such that the translations }$\lambda
_{s}:K\rightarrow K$\textit{\ are (weakly) equicontinuous, has a fixed point.%
}

\item[$(F)$:] \textit{Every nonexpansive system }$(S,(K,\mathrm{weak}))$%
\textit{, where }$K$\textit{\ is a weakly compact convex subset of a locally
convex space, such that the enveloping semigroup }$E(S,(K,\mathrm{weak}))$%
\textit{\ consists of (weakly) continuous functions, has a fixed point.}

\item[$(F^{\ast })$:] \textit{Every nonexpansive and continuous dynamical
system }$(S,(K,\mathrm{weak}))$\textit{, where }$K$\textit{\ is a weakly
compact convex subset of a locally convex space, such that the enveloping
semigroup }$E(S,(K,\mathrm{weak}))$\textit{\ consists of (weak-ly)
continuous functions, has a fixed point.}

\item[$(G)$:] \textit{Every nonexpansive dynamical system }$(S,(K,\mathrm{%
weak}))$\textit{, where }$K$\textit{\ is a weakly compact convex subset of a
locally convex space, has a fixed point.}

\item[$(G^{\ast })$:] \textit{Every nonexpansive and continuous dynamical
system }$(S,(K,\mathrm{weak}))$\textit{, where }$K$\textit{\ is a weakly
compact convex subset of a locally convex space, has a fixed point.}\medskip
\end{enumerate}

Notice that we have changed the notation slightly, partly to shorten it, and
partly because we feel that a little change in perspective may be fruitful
for future research.

\section{Semigroup actions on locally convex spaces}

The central theme in this paper is the notion of fragmentability invented by
Jayne and Rogers \cite{JaRo}. Let $(K,\omega )$ be a
topological space and let $\rho $ be a pseudometric on $K$. We say that $%
(K,\omega )$ is $(\omega ,\rho )$-fragmented if for every $\varepsilon >0$
and a nonempty set $A\subset K$ there is an $\omega $-open set $U$ in $K$
such that $U\cap A\neq \emptyset $ and $\rho $-$%
\diam%
(U\cap A)<\varepsilon $. It was proved by Namioka\ (see \cite{Na1}) that
every weakly compact subset of a Banach space is (weak,norm)-fragmented. We
need the following generalization. We say that a subset $K$ of a locally
convex space $(X,\tau )$ is $(\omega ,\tau )$-fragmented if for every $\tau $%
-open neighbourhood $V$ of $0$ and a nonempty set $A\subset K,$ there is an $%
\omega $-open set $U$ in $X$ such that $U\cap A\neq \emptyset $ and $(U\cap
A)-(U\cap A)\subset V.$ One of our main tools is the following result of
Megrelishvili\ \cite[Prop. 3.5]{Me} (see also \cite[Lemma 1.1]{GlMe}).

\begin{lemma}
\label{Me}Every weakly compact subset $K$ of a locally convex space $(X,\tau
)$ is $(\mathrm{weak},\tau )$-fragmented.
\end{lemma}

It follows that if the topology $\tau $ is determined by a family $Q$ of
seminorms on $X$ and if $K\subset X$ is weakly compact, then for every $%
\varepsilon >0,$ $p_{1},...,p_{n}\in Q$ and a nonempty set $A\subset K$,
there is a weakly open set $U$ in $X$ such that $U\cap A\neq \emptyset $ and
$p_{i}(x-y)<\varepsilon $ for every $x,y\in U\cap A$ and $i=1,...,n.$

Let $\mu $ be a probability Radon measure on a compact topological space $K.$
Recall that the support of $\mu $ is defined as the complement of the set of
points that have neighborhoods of measure $0$. It is well-known that $\mu (%
\mathrm{supp}(\mu ))=\mu (K).$ A measure $\mu $ is called $S$-invariant for
a dynamical system $(S,K)$ if $\mu (s^{-1}(A))=\mu (A)$ for every Borel set $%
A\subset K$ and $s\in S.$ The following lemma is the crucial observation for
this section.

\begin{lemma}
\label{local}Let $(S,(K,\mathrm{weak}))$ be a $Q$-nonexpansive dynamical
system, where $K$ is a minimal weakly compact $S$-invariant subset of a
locally convex space $(X,\tau )$ whose topology is determined by a family $Q$
of seminorms, and suppose that $\mu $ is an $S$-invariant Radon probability
measure on $(K,\mathrm{weak}).$ Then $K=\mathrm{supp}(\mu )$ is $\tau $%
-totally bounded.
\end{lemma}

\begin{proof}
Set $K_{0}=\mathrm{supp}(\mu )$ and notice that $\mu (s^{-1}(K_{0}))=\mu
(K_{0})=1.$ Hence $K_{0}\subset s^{-1}(K_{0})$ since $s^{-1}(K_{0})$ is
weakly closed. Similarly,
\begin{equation*}
\mu (s(K_{0}))=\mu (s^{-1}(s(K_{0})))=\mu (K_{0})=1
\end{equation*}%
and consequently, $K_{0}\subset s(K_{0}).$ Thus $s(K_{0})=K_{0}$ for every $%
s\in S$ and, since $K$ is minimal, $K=K_{0}.$ Let $p\in Q$ and $\varepsilon
>0.$ It follows from Lemma \ref{Me} that $K=\mathrm{supp}(\mu )$ is $\tau $%
-fragmented and hence there exist a weakly open set $U$ in $X$ and $x\in
U\cap \mathrm{supp}(\mu )$ such that $p(x-y)<\varepsilon $ for every $y\in
U\cap \mathrm{supp}(\mu ).$ Thus%
\begin{equation*}
\mu (\{y\in K:p(x-y)<\varepsilon \})\geq \mu (U\cap \mathrm{supp}(\mu ))>0.
\end{equation*}%
(Note that $\{y\in K:p(x-y)\leq \varepsilon \}$ is weakly compact and hence $%
\{y\in K:p(x-y)<\varepsilon \}$ is weakly Borel). Now we follow partly the
argument of \cite[Lemma 1]{Ba}. Notice that by nonexpansivity,%
\begin{equation*}
\{y\in K:p(x-y)<\varepsilon \}\subset s^{-1}(\{y\in K:p(sx-y)<\varepsilon \})
\end{equation*}%
for each $s\in S$ and, since $\mu $ is $S$-invariant,%
\begin{equation*}
\mu (\{y\in K:p(sx-y)<\varepsilon \})\geq \mu (\{y\in K:p(x-y)<\varepsilon
\})>0.
\end{equation*}%
It follows that there exists only a finite number of disjoint sets of the
form $\{y\in K:p(sx-y)<\varepsilon \},s\in S,$ which means that there exists
a finite $2\varepsilon $-net $A=\{s_{1}x,s_{2}x,...,s_{k}x\}$ of $\{sx:s\in
S\}$ with respect to $p.$ From minimality of $K,$ $K=\overline{\{sx:s\in S\}}%
^{\mathrm{weak}}$ and hence, for every $y\in K,$ there exists a net $%
\{s_{\alpha }x\}$ that converges weakly to $y.$ For each $s_{\alpha }x$ take
$z_{\alpha }\in A$ such that
\begin{equation*}
p(s_{\alpha }x-z_{\alpha })<2\varepsilon .
\end{equation*}%
Since $A$ is finite, there is a subnet $\{z_{\varphi (\beta )}\}$ of $%
\{z_{\alpha }\}$ converging to some $z\in A.$ It follows from the weak lower
semicontinuity of $p$ that%
\begin{equation*}
p(y-z)=p(w\text{-}\lim s_{\varphi (\beta )}x-w\text{-}\lim z_{\varphi (\beta
)})\leq \liminf_{\alpha }p(s_{\varphi (\beta )}x-z_{\varphi (\beta )})\leq
2\varepsilon .
\end{equation*}%
Thus for every $\varepsilon >0$ and $p\in Q$ there exists a finite $%
\varepsilon $-net for $K$ and it means that $K$ is $\tau $-totally bounded.
\end{proof}

We can now give a proof of a fixed point theorem that combines Lemma \ref%
{local} and the arguments in \cite[Theorem 4.1]{La} (see also \cite[Theorem
5.3]{LaTa}). Denote by $C(K)$ the space of weakly continuous complex-valued
functions defined on a weakly compact set $K.$

\begin{theorem}
\label{Main}Let $(S,(K,\mathrm{weak}))$ be a $Q$-nonexpansive dynamical
system, where $K$ is a weakly compact convex subset of a locally convex
space $(X,\tau )$ whose topology is determined by a family $Q$ of seminorms
and let $Y$ be a closed linear subspace of $\ell ^{\infty }(S)$ containing
constants and invariant under translations. Suppose that the function $S\ni
s\rightarrow f_{y}(s)=f(sy)$ belongs to $Y$ for every $y\in K$ and every $%
f\in C(K).$ If $Y$ has a left invariant mean then $(S,(K,\mathrm{weak}))$
has a fixed point.
\end{theorem}

\begin{proof}
It follows from Kuratowski-Zorn's lemma that there exists a minimal weakly
compact and convex subset $C$ of $K$ which is invariant under $S.$ Let $F$
be a minimal weakly compact subset of $C$ which is invariant under $S.$
Notice that $f_{y}\in Y$ for every $y\in F$ and $f\in C(F).$ Fix $y\in F$
and let $m$ be a left invariant mean on $Y.$ Define a positive functional $%
\Phi $ on $C(F)$ by%
\begin{equation*}
\Phi (f)=m(f_{y})
\end{equation*}%
for $f\in C(F).$ Let $_{t}f(x)=f(tx)$ for every $t\in S,\ x\in F$ and $f\in
C(F).$ Then $_{t}f:F\rightarrow \mathbb{C}$ is weakly continuous and $\Phi
(f)=\Phi (_{t}f).$ Let $\mu $ be the probability (weakly Borel) Radon
measure on $F$ corresponding to $\Phi .$ Then $\mu (A)=\mu (s^{-1}(A))$ for
every weakly Borel subset $A$ of $F$ and $s\in S.$ It follows from Lemma \ref%
{local} that $F=\mathrm{supp}(\mu )$ is $\tau $-totally bounded. We can
assume without loss of generality that $(X,\tau )$ is complete (otherwise,
consider the closure of $F$ in the completion of $X$). Thus $F$ is $\tau $%
-compact.

Suppose that $F$ is not a singleton. Then there is a seminorm $p\in Q$ such
that $r=\sup \{p(x-y):x,y\in F\}>0.$ By a counterpart of \cite[Lemma 1]{De}
(applied to the seminorm $p$ instead of a norm), there is $u\in \overline{%
\co%
}F\subset C$ such that $r_{0}=\sup \{p(u-y):y\in F\}<r.$ Let
\begin{equation*}
C_{0}=\{x\in C:p(x-y)\leq r_{0}\text{ for all }y\in F\}.
\end{equation*}%
Then $u\in C_{0}$ and $C_{0}$ is a weakly compact convex proper subset of $%
C. $ Since the system $(S,K)$ is $Q$-nonexpansive and $s(F)=F,s\in S,$ we
have $s(C_{0})\subset C_{0}$ for each $s\in S$ which contradicts the
minimality of $C_{0}.$ Thus $K$ consists of a single point $x$ and $sx=x$
for every $s\in S.$
\end{proof}

As alluded to in the introduction, Hsu \cite{Hsu} proved property $(G)$ for
a left reversible and discrete semigroup $S$, and Lau and Zhang \cite[%
Theorem 5.4]{LaZh1} generalized Hsu's result to left reversible, metrizable
semitopological semigroups by showing property $(G^{\ast })$. (Notice that
for discrete semigroups, properties $(G$) and $(G^{\ast })$ are equivalent).
Having Theorem \ref{Main}, we can extend the above theorems to left amenable
semitopological semigroups and also drop the separability assumption from
some results in \cite{LaZh1}, thus giving a full characterization for the
existence of a left invariant mean on $AP(S),WAP(S)$ or $WAP(S)\cap LUC(S)$,
respectively, in terms of the fixed point property of the semitopological
semigroup $S$ acting nonexpansively on a weakly compact convex subset of a
locally convex space.

\begin{theorem}
\label{LUC}Let $S$ be a semitopological semigroup. If $LUC(S)$ has a LIM,
then $S$ has property $(G^{\ast })$.
\end{theorem}

\begin{proof}
Let $K$ be a weakly compact convex subset of a locally convex space whose
topology is determined by a family $Q$ of seminorms. If $(S,(K,\mathrm{weak}%
))$ is continuous in weak topology, then by definition, the action $\pi
:S\times K\rightarrow K$ is jointly weakly continuous and it follows from
\cite[Lemma 5.1]{LaTa} that the function $f_{y}(s)=f(sy),s\in S,$ belongs to
$LUC(S)$ for every $f\in C(K)$ and $y\in K$. Since the action is also $Q$%
-nonexpansive and $LUC(S)$ has a left invariant mean, the result follows
from Theorem \ref{Main} specialized to $Y=LUC(S)$.
\end{proof}

Theorem \ref{LUC} extends also \cite[Corollary 3.4]{Sa} in two aspects: from
strongly amenable to amenable semigroups and from weakly compact convex sets
in Banach spaces to locally convex spaces.

For property $(G)$ we have the following result which extends (the right
part of) \cite[Theorem 3.1]{ADN} from Banach spaces to locally convex spaces.

\begin{theorem}
\label{LMC}If $LMC(S)$ has a LIM, then $S$ has property $(G)$.
\end{theorem}

\begin{proof}
Given a $Q$-nonexpansive system $(S,(K,\mathrm{weak})),$ the action\ $\pi
:S\times K\rightarrow K$ is weakly separately continuous and the same
reasoning as in the proof of Theorem 3 in \cite{Mi} yields $f_{y}\in LMC(S)$
for every $f\in C(K)$ and $y\in K.$ Now the result follows from Theorem \ref%
{Main} specialized to $Y=LMC(S)$.
\end{proof}

It appears that $LMC(S)$ should be replaced by $WLUC(S)$ in Theorem \ref{LMC}
as in the case of affine actions (see \cite[Theorem 4]{Mi}) but it is out of
our reach for now. In fact, it seems to be an open problem whether $WLUC(S)$
is a proper subspace of $LMC(S).$

Since property $(G)$ implies $WAP(S)$ has a LIM, it follows from Hsu's
theorem that if $S$ is discrete and left reversible, then $WAP(S)$ has a
LIM. It improves an earlier result of Ryll-Nardzewski who proved the
existence of LIM on $WAP(G)$ when $G$ is a (discrete) group. Lau and Zhang
\cite[Theorem 3.4]{LaZh1} showed that if $S$ is a separable semitopological
semigroup, then $WAP(S)$ has a LIM iff $S$ has property $(F)$, and one of
the main questions in \cite{LaZh2} was about a similar characterization for
any semitopological semigroup. We answer this question in the affirmative.

\begin{theorem}
\label{WAP}$WAP(S)$ has a LIM if and only if $S$ has property $(F)$.
\end{theorem}

\begin{proof}
Assume that $WAP(S)$ has a LIM and a system $(S,(K,\mathrm{weak}))$ is $Q$%
-nonexpansive, where $K$ is a weakly compact convex subset of a locally
convex space, and such that the enveloping semigroup $E(S,K)$ consists of
weakly continuous functions. It follows from \cite[Lemma 3.2]{LaZh1} that $%
f_{y}\in WAP(S)$ for every $f\in C(K)$ and $y\in K.$ Thus the assumptions of
Theorem \ref{Main} are satisfied with $Y=WAP(S)$ and we obtain a fixed point
of $(S,K)$. The reverse implication follows in the same way as in \cite[%
Theorem 3.4]{LaZh1}.
\end{proof}

As a by-product, combining Theorem \ref{LUC} with Theorem \ref{WAP} we have $%
(G^{\ast })\not\Rightarrow (G)$ for, otherwise, amenability of $LUC(S)$
would imply amenability of $WAP(S)$ which is in general not the case (see
\cite[Example 5.5]{LaZh1}). Our next result drops the separability
assumption from \cite[Theorem 5.1]{LaZh1} and thus provides a full
characterization of a semigroup $S$ that has a left invariant mean on $%
WAP(S)\cap LUC(S)$ in terms of a fixed point property for nonexpansive
mappings.

\begin{theorem}
\label{WAPLUC}$WAP(S)\cap LUC(S)$ has a LIM if and only if $S$ has property $%
(F^{\ast })$.
\end{theorem}

\begin{proof}
Assume that $WAP(S)\cap LUC(S)$ has a LIM, $K$ is a weakly compact convex
subset of a locally convex space and a system $(S,(K,\mathrm{weak}))$ is $Q$%
-nonexpansive, weakly continuous and the enveloping semigroup $E(S,K)$
consists of weakly continuous functions. Thus the action $\pi :S\times
K\rightarrow K$ is jointly weakly continuous and hence $f_{y}\in LUC(S)$ for
every $f\in C(K)$ and $y\in K.$ Furthermore, as in the proof of Theorem \ref%
{WAP}, $f_{y}\in WAP(S).$ Applying Theorem \ref{Main} with $Y=WAP(S)\cap
LUC(S)$ we get a fixed point of $S$ in $K.$ The reverse implication follows
in the same way as in \cite[Theorem 5.1]{LaZh1}.
\end{proof}

Finally, we consider the space of almost periodic functions $AP(S).$ Lau
\cite[Theorem 4.1]{La} characterized amenability of $AP(S)$ in terms of a
fixed point property for nonexpansive mappings acting on a compact convex
subset of a locally convex space. He showed that $AP(S)$ has a LIM if and
only if $S$ has property $(D)$. In \cite[Theorem 3.9]{LaZh1}, the authors
characterized amenability of $AP(S)$ in terms of a fixed point property on
weakly compact convex sets and proved that $AP(S)$ has a LIM if and only if $%
S$ has property $(E^{\prime })$. Furthermore, they proved that $AP(S)$ has a
LIM if and only if $S$ has property $(E)$ provided $S$ is separable (see
\cite[Theorem 3.6]{LaZh1}). We have such a characterization for any
semigroup.

\begin{theorem}
\label{AP}$AP(S)$ has a LIM if and only if $S$ has property $(E)$.
\end{theorem}

\begin{proof}
Suppose that $AP(S)$ has a LIM, $K$ is a weakly compact convex subset of a
locally convex space and a system $(S,(K,\mathrm{weak}))$ is $Q$%
-nonexpansive and such that the translations $\lambda _{s}:K\ni x\rightarrow
sx\in K$ are equicontinuous in weak topology. Then $f_{y}\in AP(S)$ for
every $f\in C(K)$ and $y\in K$ by \cite[Lemma 3.1]{La}. Applying Theorem \ref%
{Main} with $Y=AP(S)$ we get a fixed point of $S$ in $K.$ The reverse
implication holds by \cite[Theorem 3.9]{LaZh1}.
\end{proof}

Thus it is shown that properties $(D),(E^{\prime })$ and $(E)$ are
equivalent that solves another open problem from \cite[p. 2542]{LaZh1} (see
also \cite[Problem 1]{LaZh4}).

One aspect of this program that has not been studied yet is its relation to
the Bruck retraction method developed in \cite{Br1, Br2}. We use the
following consequence of Bruck's theorem \cite[Theorem 3]{Br2} to get the
qualitative information about the structure of the set of fixed points of $%
(S,K).$

\begin{theorem}
\label{Bruck}Let $K$ be a compact Hausdorff topological space and $S$ a
(discrete) semigroup of mappings on $K$ (the separate continuity of the
operation is not required). Suppose that $S$ is compact in the product
topology of $K^{K}$ and each nonempty closed $S$-invariant subset of $K$
contains a fixed point of $S$. Then there exists in $S$ a retraction of $K$
onto $F(S)=\{x\in K:sx=x$ for every $s\in S\}.$
\end{theorem}

Note that a retraction here is a mapping $r:K\rightarrow F(S)$ such that $%
r\circ r=r$ (the continuity of $r$ in the topology of $K$ is not required).
We sketch the proof for the convenience of the reader.

\begin{proof}
Following \cite[Theorem 3]{Br2}, we will construct a one-element left ideal $%
\{e\}$ of $S.$ By Kuratowski-Zorn's lemma there exists a minimal left ideal $%
J$ of $S$ that is compact in the product topology of $K^{K}$. If $x\in K$
then $Jx=\{sx:s\in J\}$ is compact as the image of the compact set $J$ under
the continuous projection $S\ni s\rightarrow sx\in K$ and $S$-invariant
since $sJ\in J$ for each $s\in S.$ By assumption, there is $u\in Jx$ such
that $su=u$ for $s\in S.$ Define $I=\{s\in J:sx=u\}.$ Then $\emptyset \neq
I\subset J$ is a left ideal of $S$ and compact in product topology. From
minimality of $J,$ $I=J$, that is, $sx=u$ for every $s\in J.$ Since $x$ is
arbitrary, $J$ consists of a single element $e:K\rightarrow K.$ Thus $se=e$
for every $s\in S$ and consequently $ex\in F(S)$ for every $x\in K.$
Moreover, $ex=x$ for $x\in F(S)$ since $e\in S.$ It shows that $e$ is a
retraction of $K$ onto $F(S).$
\end{proof}

The following theorem is a qualitative complement to the results of this
section.

\begin{theorem}
\label{q}Let $S$ be a semitopological semigroup that satisfies one of
properties from $(D)$ to $(G^{\ast })$. Then the set of fixed points of $%
(S,K)$ is a $Q$-nonexpansive retract of $K$.
\end{theorem}

\begin{proof}
Suppose that $S$ satisfies one of these properties. Put $\hat{S}%
=\{T:K\rightarrow K\mid T$ is $Q$-nonexpansive and $F(S)\subset F(T)\}.$
Notice that $K^{K}$ is compact in the topology of pointwise convergence when
$K$ is given the weak topology (or $\tau $-topology in case of property $(D)$%
). Furthermore, $\hat{S}\subset K^{K}$ is closed in this topology since
\begin{equation*}
p(w\text{-}\lim T_{\alpha }x-w\text{-}\lim T_{\alpha }y)\leq \liminf_{\alpha
}p(T_{\alpha }x-T_{\alpha }y)\leq p(x-y)
\end{equation*}%
for every $x,y\in K,p\in Q$ and a convergent net $\{T_{\alpha }\}\subset
\hat{S}.$ Thus $\hat{S}$ is compact in this topology and $F(S)=F(\hat{S}).$
Let $K_{0}$ be a closed $\hat{S}$-invariant subset of $K$ and select $x\in
K_{0}.$ Then $\hat{K}=\{Tx:T\in \hat{S}\}$ is a compact convex $\hat{S}$%
-invariant subset of $K_{0}$. Since $S$ satisfies one of properties from $%
(D) $ to $(G^{\ast })$, there is a fixed point of $\hat{S}$ in $\hat{K}%
\subset K_{0}.$ From Theorem \ref{Bruck} there is in $\hat{S}$ a retraction
of $K$ onto $F(\hat{S})=F(S).$ This completes the proof since every element
in $\hat{S}$ is $Q$-nonexpansive.
\end{proof}

\section{Semigroup actions on Banach spaces}

An old problem in fixed point theory that dates back to the 1970s (see \cite%
{LaZh2} for a discrete case and \cite{La2} for a general case), posed by A.
T.-M. Lau, concerns a counterpart of the well-known Day-Mitchell
\textquotedblleft affine\textquotedblright\ characterization of amenable
semigroups: does a semitopological semigroup $S$ have the fixed point
property:\medskip

\begin{enumerate}
\item[$(F_{\ast })$:] \textit{Every nonexpansive and continuous dynamical
system }$(S,(K,\mathrm{weak}^{\ast }))$\textit{, where }$K$\textit{\ is a
weak}$^{\ast }$\textit{\ compact convex subset of a dual Banach space, has a
fixed point}
\end{enumerate}

\medskip \noindent if $LUC(S)$ has a LIM? Partial solutions to this problem
were obtained for separable weak$^{\ast }$\ compact convex sets\textit{\ }in
\cite{LaZh2} and for commutative semigroups in \cite{BoWi}. We refer the
reader to \cite{LaZh4} for a discussion and further references. In this
section we extend the techniques from Section 3 to propose a new approach to
Lau's problem.

Suppose $K$ is a weak$^{\ast }$ compact subset of a dual Banach space and
denote by $\mathcal{B}(K,\mathrm{weak}^{\ast })$ the sigma-algebra of weak$%
^{\ast }$ Borel subsets of $K.$ In general, the sigma-algebra $\mathcal{B}%
(K,\left\Vert \cdot \right\Vert )$ of norm Borel sets may be larger than $%
\mathcal{B}(K,\mathrm{weak}^{\ast })$ but the closed balls in $K$ are weak$%
^{\ast }$ compact and hence are weak$^{\ast }$ Borel. Let $\mu $ be a
probability measure on the sigma-algebra $\mathcal{B}_{0}(K,\left\Vert \cdot
\right\Vert )\subset \mathcal{B}(K,\mathrm{weak}^{\ast })$ generated by the
family of balls in $K$ and define
\begin{equation*}
\mathrm{supp}_{\left\Vert \cdot \right\Vert }(\mu )=\{x\in K:\mu (\{y\in
K:\left\Vert x-y\right\Vert <\varepsilon \})>0\text{ for each }\varepsilon
>0\}.
\end{equation*}%
Recall from Section 2 that by $(S,(K,\left\Vert \cdot \right\Vert ))$ we
mean a dynamical system, where $\pi :S\times K\rightarrow K$ is separately
continuous and $K$ is considered with the norm-topology.

\begin{lemma}
\label{banach}Let $(S,(K,\left\Vert \cdot \right\Vert ))$ be a nonexpansive
dynamical system, where $K$ is a minimal weak$^{\ast }$ compact $S$%
-invariant subset of a dual Banach space and suppose that $\mu $ is a
probability $S$-invariant measure on $\mathcal{B}_{0}(K,\left\Vert \cdot
\right\Vert ).$ If $\mathrm{supp}_{\left\Vert \cdot \right\Vert }(\mu )\neq
\emptyset $ then $K$ is norm-compact.
\end{lemma}

\begin{proof}
By assumption, there is $x\in K$ such that $\mu (\{y\in K:\left\Vert
x-y\right\Vert <\varepsilon \})>0$ for each $\varepsilon >0.$ Thus, by
nonexpansivity,%
\begin{equation*}
\mu (s^{-1}(\{y\in K:\left\Vert sx-y\right\Vert <\varepsilon \}))\geq \mu
(\{y\in K:\left\Vert x-y\right\Vert <\varepsilon \})>0
\end{equation*}%
and, since $\mu $ is $S$-invariant,%
\begin{equation*}
\mu (\{y\in K:\left\Vert sx-y\right\Vert <\varepsilon \})=\mu (s^{-1}(\{y\in
K:\left\Vert sx-y\right\Vert <\varepsilon \}))>0
\end{equation*}%
for each $s\in S$ and $\varepsilon >0.$ It follows that there exists a $%
2\varepsilon $-net $A=\{s_{1}x,s_{2}x,...,s_{k}x\}$ of $\{sx:s\in S\}$ for
every $\varepsilon >0.$ Hence $\{sx:s\in S\}$ is norm-totally bounded. From
minimality of $K,$ $K=\overline{\{sx:s\in S\}}^{\mathrm{weak}^{\ast }}=%
\overline{\{sx:s\in S\}}^{\mathrm{norm}}$ is norm-compact.
\end{proof}

Notice that in particular, Lemma \ref{banach} holds if $\mu $ is a
probability $S$-invariant Radon measure on $\mathcal{B}(K,\left\Vert \cdot
\right\Vert )$ that may be of independent interest.

There is a large class of sets for which $\mathrm{supp}_{\left\Vert \cdot
\right\Vert }(\mu )\neq \emptyset $ for every Radon measure $\mu $ on $%
\mathcal{B}(K,\mathrm{weak}^{\ast }).$ Recall that a convex closed subset $C$
of a Banach space $X$ has the Radon-Nikod\'{y}m property (RNP for short) if
for any measure space $(\Omega ,\mathcal{F},\mu )$ and an $X$-valued measure
$F:\mathcal{F}\rightarrow X$ such that $\{F(B)/\mu (B):B\in \mathcal{F},\mu
(B)>0\}\subset C,$ there is a Bochner integrable function $\varphi :\Omega
\rightarrow C$ such that $F(B)=\int\nolimits_{B}\varphi \ d\mu $ for each $%
B\in \mathcal{F}$.

Suppose that $K$ is a weak$^{\ast }$ compact convex subset with the RNP of a
dual Banach space and $\mu $ a Radon probability measure on $\mathcal{B}(K,%
\mathrm{weak}^{\ast }).$ Let $\mathrm{supp}(\mu )$ denote the support of $%
\mu $ with respect to $\mathcal{B}(K,\mathrm{weak}^{\ast })$. By the results
of Michael, Namioka, Phelps and Stegall, the identity map $\mathrm{id}:(%
\mathrm{supp}(\mu ),\mathrm{weak}^{\ast })\rightarrow (\mathrm{supp}(\mu ),%
\mathrm{norm})$ has a point of continuity $x$ (see, e.g., \cite[Theorem
4.2.13]{Bo}). It follows that for every $\varepsilon >0$ there is a weak$%
^{\ast }$ open neighbourhood $U$ of $x$ such that $\left\Vert x-y\right\Vert
<\varepsilon $ for each $y\in U\cap \mathrm{supp}(\mu ).$ Hence
\begin{equation*}
\mu (\{y\in K:\left\Vert x-y\right\Vert <\varepsilon \})\geq \mu (U\cap
\mathrm{supp}(\mu ))>0
\end{equation*}%
and therefore $x\in \mathrm{supp}_{\left\Vert \cdot \right\Vert }(\mu ).$
Thus we obtain a counterpart of Theorem \ref{Main}.

\begin{theorem}
\label{Main star}Let $(S,(K,\mathrm{weak}^{\ast }))$ be a nonexpansive
dynamical system, where $K$ is a weak$^{\ast }$ compact convex subset with
the RNP of a dual Banach space and let $Y$ be a closed linear subspace of $%
\ell ^{\infty }(S)$ containing constants and invariant under translations.
Suppose that the function $S\ni s\rightarrow f_{y}(s)=f(sy)$ belongs to $Y$
for every $y\in K$ and weak$^{\ast }$ continuous function $f:K\rightarrow
\mathbb{C}.$ If $Y$ has a left invariant mean then $(S,(K,\mathrm{weak}%
^{\ast }))$ has a fixed point.
\end{theorem}

\begin{proof}
Let $C$ be a minimal weak$^{\ast }$ compact convex $S$-invariant subset of $%
K $ and $F$ a minimal weak$^{\ast }$ compact $S$-invariant subset of $C.$ As
in the proof of Theorem \ref{Main}, choose a left invariant mean on $Y$, fix
$y\in F$ and set $\Phi (f)=m(f_{y})$ for $f\in C(F)$, the space of weak$%
^{\ast }$ continuous complex-valued functions on $F$. Then $\mu (A)=\mu
(s^{-1}(A))$ for every weak$^{\ast }$ Borel subset $A$ of $F$ and $s\in S$,
where $\mu $ is the probability (weak$^{\ast }$ Borel) Radon measure on $F$
corresponding to $\Phi .$ Since $F$ is a weak$^{\ast }$ compact subset of
the set $K$ with the RNP, $\mathrm{supp}_{\left\Vert \cdot \right\Vert }(\mu
)\neq \emptyset $ and it follows from Lemma \ref{banach} that $F$ is
norm-compact. Furthermore, from minimality, $F=\mathrm{supp}(\mu )$ and
hence $s(F)=F$ for every $s\in S.$ As in the proof of Theorem \ref{Main}, $%
K=F$ consists of a single point $x$ and thus $sx=x$ for every $s\in S.$
\end{proof}

In particular, we obtain a partial solution to Lau's problem.

\begin{theorem}
\label{LUC star}Let $S$ be a semitopological semigroup and $(S,(K,\mathrm{%
weak}^{\ast }))$ a nonexpansive continuous dynamical system, where $K$ is a
weak$^{\ast }$ compact convex subset with the RNP of a dual Banach space. If
$LUC(S)$ has a LIM, then $S$ has a fixed point in $K$ and the set $F(S)$ of
fixed points is a nonexpansive retract of $K$.
\end{theorem}

\begin{proof}
It follows from \cite[Lemma 5.1]{LaTa} that $f_{y}\in LUC(S)$ for every weak$%
^{\ast }$ continuous function $f:K\rightarrow \mathbb{C}$ and $y\in K.$ By
Theorem \ref{Main star}, there is a fixed point of $S$ in $K.$

Put $\hat{S}=\{T:K\rightarrow K\mid T$ is nonexpansive and $F(S)\subset
F(T)\}$ and notice that $\hat{S}\subset K^{K}$ is closed in the topology of
weak$^{\ast }$ pointwise convergence since
\begin{equation*}
\left\Vert w^{\ast }\text{-}\lim T_{\alpha }x-w^{\ast }\text{-}\lim
T_{\alpha }y\right\Vert \leq \liminf_{\alpha }\left\Vert T_{\alpha
}x-T_{\alpha }y\right\Vert \leq \left\Vert x-y\right\Vert
\end{equation*}%
for every $x,y\in K$ and every weak$^{\ast }$ convergent net $\{T_{\alpha
}\}\subset \hat{S}.$ Thus $\hat{S}$ is compact in this topology. Let $K_{0}$
be a weak$^{\ast }$ closed $\hat{S}$-invariant subset of $K.$ Choose $x\in
K_{0}$ and notice that $\hat{K}=\{Tx:T\in \hat{S}\}$ is a weak$^{\ast }$
compact convex $\hat{S}$-invariant subset of $K_{0}$. By the first part of
this theorem there is a fixed point of $S$ in $\hat{K}\subset K_{0}.$ But $%
F(S)=F(\hat{S})$ and it follows from Theorem \ref{Bruck} that there exists
in $\hat{S}$ a retraction of $K$ onto $F(S).$
\end{proof}

As a consequence, Theorem \ref{LUC star} holds for any weak$^{\ast }$
compact convex subset of the dual of an Asplund space, in particular for any
norm separable weak$^{\ast }$ compact convex subset of a dual Banach space
(see, e.g., \cite[Theorem 2]{Na4}). We leave to the reader to formulate and
prove appropriate versions of the above theorem when $AP(S),WAP(S),WAP(S)%
\cap LUC(S)$ or $LMC(S)$ have a LIM, respectively.

There is another class of sets, related to the recent Bader-Gelander-Monod
theorem (see \cite[Theorem A]{BGM}), for which Lau's problem has an
affirmative solution. Recall that a Banach space $X$ is said to be $L$%
-embedded if its bidual $X^{\ast \ast }$ can be decomposed as $X^{\ast \ast
}=X\oplus _{1}X_{s}$ for some $X_{s}\subset X^{\ast \ast }$ (with the $\ell
^{1}$-norm). A Banach space $X$ is $M$-embedded if $X$ is an $M$-ideal in
its bidual $X^{\ast \ast }.$ It is known that if $X$ is $M$-embedded, then $%
X^{\ast }$ is $L$-embedded and the converse is not true in general (see \cite%
{HWW}). Examples of $L$-embedded Banach spaces include all $L_{1}$ spaces,
preduals of von Neumann algebras and the Hardy space $H_{1}$. In turn, $%
c_{0}(\Gamma )$ and $K(H)$, the Banach space of all compact operators on a
Hilbert space $H$, are examples of $M$-embedded spaces. The following notion
was introduced in \cite{LaZh2}.

\begin{definition}
Let $C$ be a nonempty subset of a Banach space $X$ and denote by $\overline{C%
}^{\mathrm{weak}^{\ast }}$ the closure of $C$ in $X^{\ast \ast }$ in the weak%
$^{\ast }$ topology of $X^{\ast \ast }$. We say that $C$ is $L$-embedded if
there is a subspace $X_{s}$ of $X^{\ast \ast }$ such that $X\oplus
_{1}X_{s}\subset X^{\ast \ast }$ and $\overline{C}^{\mathrm{weak}^{\ast
}}\subset C\oplus _{1}X_{s}.$
\end{definition}

It was proved in \cite{LaZh2} that every $L$-embedded set is weakly closed.
Moreover, a Banach space is $L$-embedded iff its unit ball is $L$-embedded.
Notice that a weakly compact subset $C$ of any Banach space $X$ is $L$%
-embedded since $\overline{C}^{\mathrm{weak}^{\ast }}=C.$

If $A,C$ are subsets of a Banach space $X$ with $A$ bounded, we define the
Chebyshev radius of $A$ in $C$ by%
\begin{equation*}
r_{C}(A)=\inf_{x\in C}\sup_{y\in A}\left\Vert x-y\right\Vert
\end{equation*}%
and the Chebyshev center of $A$ in $C$ by%
\begin{equation*}
E_{C}(A)=\{x\in C:\sup_{y\in A}\left\Vert x-y\right\Vert =r_{C}(A)\}.
\end{equation*}

\begin{lemma}[{see {\protect\cite[Lemma 3.3]{LaZh2}}}]
\label{embed}Let $C$ be an $L$-embedded subset of a Banach space $X$ and $A$
a bounded subset of $X$. Then the Chebyshev center $E_{C}(A)$ is weakly
compact.
\end{lemma}

If we combine Theorem \ref{LUC} and Lemma \ref{embed} we have the following
complement of \cite[Theorem 3.11]{LaZh2}, where a similar statement was
proved for metrizable left reversible semitopological semigroups.

\begin{theorem}
\label{L-em}Let $C$ be a bounded convex $L$-embedded subset of a Banach
space $X$ and let $(S,(C,\mathrm{weak}))$ be a nonexpansive and continuous
dynamical system. If $LUC(S)$ has a LIM and $C$ contains a bounded subset $A$
such that $s(A)=A$ for all $s\in S$, then there is a fixed point of $S$ in $%
E_{C}(A).$
\end{theorem}

\begin{proof}
Notice that $E_{C}(A)$ is convex, $s(E_{C}(A))\subset E_{C}(A)$ and it
follows from Lemma \ref{embed} that $E_{C}(A)$ is weakly compact. Now the
result follows from Theorem \ref{LUC}.
\end{proof}

In a similar way we can prove the theorem that drops the separability
assumption from \cite[Theorem 3.16]{LaZh2}.

\begin{theorem}
\label{W-em}Let $C$ be a bounded convex $L$-embedded subset of a Banach
space $X$ and let $(S,(C,\mathrm{weak}))$ be a nonexpansive dynamical system
\textit{such that the enveloping semigroup }$E(S,(C,\mathrm{weak}))$\textit{%
\ consists of (weakly) continuous functions}. If $WAP(S)$ has a LIM and $C$
contains a bounded subset $A$ such that $s(A)=A$ for all $s\in S$, then
there is a fixed point of $S$ in $E_{C}(A).$
\end{theorem}

\begin{proof}
As before, $s(E_{C}(A))\subset E_{C}(A)$ for each $s\in S$ and $E_{C}(A)$ is
convex and weakly compact. Now the result follows from Theorem \ref{WAP}.
\end{proof}

Theorems \ref{L-em} and \ref{W-em} are related to Theorem A in the recent
paper of Bader, Gelander and Monod \cite{BGM} that was used to give a short
proof of the long-standing derivation problem for a convolution algebra $%
L^{1}(G)$ of a locally compact group $G.$ Notice that we have also
appropriate versions of the above theorems when $AP(S),WAP(S)\cap LUC(S)$ or
$LMC(S)$ have a LIM, as in the case of Theorem \ref{LUC star}.

Our final theorem concerns weak$^{\ast }$ compact convex sets in the dual of
an M-embedded Banach space. Lemma 3.2 in \cite{LaZh2} asserts that any weak$%
^{\ast }$ closed subset of the dual space of an $M$-embedded Banach space is
$L$-embedded. Therefore, we can use the results of Section 3 again.

\begin{theorem}
\label{M}Let $S$ be a semitopological semigroup and $(S,(K,\mathrm{weak}%
^{\ast }))$ a nonexpansive dynamical system, where $K$ is a weak$^{\ast }$
compact convex subset of the dual space $X^{\ast }$ of an $M$-embedded
Banach space $X.$ Suppose that one of the following conditions holds:
\end{theorem}

\begin{enumerate}
\item[(i)] $AP(S)$ has a LIM and the translations $\lambda _{s}:K\rightarrow
K$\ are weak$^{\ast }$ equicontinuous

\item[(ii)] $WAP(S)$ has a LIM and the enveloping semigroup $E(S,K)$\
consists of weak$^{\ast }$ continuous functions

\item[(iii)] $LMC(S)$ has a LIM

\item[(iv)] $WAP(S)\cap LUC(S)$ has a LIM, the enveloping semigroup $E(S,K)$%
\ consists of weak$^{\ast }$ continuous functions and the action $\pi
:S\times K\rightarrow K$ is jointly continuous when $K$ is given weak$^{\ast
}$ topology

\item[(v)] $LUC(S)$ has a LIM and the action $\pi :S\times K\rightarrow K$
is jointly continuous when $K$ is given weak$^{\ast }$ topology.
\end{enumerate}

Then there is a fixed point of $S$ in $K$ and the set $F(S)$ of fixed points
is a nonexpansive retract of $K$.

\begin{proof}
All the proofs follow the same pattern, so we only prove (i). Since the
translations $\lambda _{s}:K\rightarrow K$\ are weak$^{\ast }$
equicontinuous, it follows from \cite[Lemma 3.1]{La} that the function $%
f_{y}(s)=f(sy)\in AP(S)$ for every weak$^{\ast }$ continuous $f:K\rightarrow
\mathbb{C}$ and $y\in K.$ By assumption, there is a left invariant mean $m$
on $AP(S)$, fix $y\in K$, and define a positive functional $\Phi $ on $C(K)$%
, the space of weak$^{\ast }$ continuous complex-valued functions on $K,$ by%
\begin{equation*}
\Phi (f)=m(f_{y})
\end{equation*}%
for $f\in C(K).$ Let $\mu $ be the probability $S$-invariant weak$^{\ast }$
Borel Radon measure on $K$ corresponding to $\Phi $ and denote by $K_{0}=%
\mathrm{supp}(\mu )$ its (weak$^{\ast }$) support. Then $s(K_{0})=K_{0}$ and
hence $s(E_{K}(K_{0}))\subset E_{K}(K_{0})$ for each $s\in S.$ Moreover, by
\cite[Lemma 3.2]{LaZh2}, $K$ is $L$-embedded and hence $E_{K}(K_{0})$ is
weakly compact by Lemma \ref{embed}. We show that $\{\lambda _{s}\}_{s\in S}$
are weakly equicontinuous on $E_{K}(K_{0}).$ Otherwise, there exist a weak
neighbourhood $V$ of $0$ and the nets $\{\lambda _{s_{\alpha }}\}$ and $%
\{x_{\alpha }\},\{y_{\alpha }\}\subset E_{K}(K_{0})$ such that $w$-$\lim
(x_{\alpha }-y_{\alpha })=0$ and $s_{\alpha }x_{\alpha }-s_{\alpha
}y_{\alpha }\notin V.$ Hence $w^{\ast }$-$\lim (x_{\alpha }-y_{\alpha })=0$
and from weak$^{\ast }$ equicontinuity of $\{\lambda _{s}\},$ $w^{\ast }$-$%
\lim (s_{\alpha }x_{\alpha }-s_{\alpha }y_{\alpha })=0.$ Since $E_{K}(K_{0})$
is weakly compact, there are subnets $\{\lambda _{s_{\varphi (\beta )}}\}$, $%
\{x_{\varphi (\beta )}\}$, $\{y_{\varphi (\beta )}\}$ such that $s_{\varphi
(\beta )}x_{\varphi (\beta )}-s_{\varphi (\beta )}y_{\varphi (\beta )}$
converges weakly and hence also weak$^{\ast }$ to $z\neq 0$, and we obtain a
contradiction. In a similar way, we show that the action $\pi :S\times
E_{K}(K_{0})\rightarrow E_{K}(K_{0})$ is separately continuous when $K$ is
given the weak topology, that is, $(S,(E_{K}(K_{0}),\mathrm{weak}))$ is a
dynamical system. Now we can apply Theorem \ref{AP} to get a fixed point of $%
S$ in $E_{K}(K_{0}).$ As in the proof of Theorem \ref{LUC star}, the set $%
F(S)$ of fixed points is a nonexpansive retract of $K$.
\end{proof}

Theorem \ref{M} (v) gives a partial solution to Lau's problem and
complements \cite[Theorem 3.18]{LaZh2}, where a similar statement was proved
for [metrizable] left reversible semitopological semigroups.

If we analyse Lemma \ref{banach} and the proof of the above theorem we
conclude that it holds for any weak$^{\ast }$ compact convex subset $K$ of a
dual Banach space provided there exists a weak$^{\ast }$ Borel $S$-invariant
Radon probability measure $\mu $ on each minimal weak$^{\ast }$ compact $S$%
-invariant subset of $K$ such that $\mathrm{supp}_{\left\Vert \cdot
\right\Vert }(\mu )\neq \emptyset .$ Therefore, the following question is
central for our approach to Lau's problem described at the beginning of this
section:

\begin{quest}
Does there exist a weak$^{\ast }$ compact set $K$ and a Radon probability
measure on the sigma-algebra $\mathcal{B}(K,\mathrm{weak}^{\ast })$ of weak$%
^{\ast }$ Borel subsets of $K$ such that $\mathrm{supp}_{\left\Vert \cdot
\right\Vert }(\mu )=\emptyset $?
\end{quest}

A negative answer to this question yields a complete, affirmative solution
to Lau's problem.


\begin{thebibliography}{99}
\bibitem{ADN} A. Aminpour, A. Dianatifar, R. Nasr-Isfahani, Asymptotically
non-expansive actions of strongly amenable semigroups and fixed points, J.
Math. Anal. Appl. 461 (2018), 364--377.

\bibitem{BGM} U. Bader, T. Gelander, N. Monod, A fixed point theorem for $%
L_{1}$ spaces, Invent. Math. 189 (2012), 143--148.

\bibitem{Ba} W. Bartoszek, Nonexpansive, $\mathcal{T}$-continuous
antirepresentations have common fixed points, Proc. Amer. Math. Soc. 127
(1999), 1051--1055.

\bibitem{BoWi} S. Borzdy\'{n}ski, A.Wi\'{s}nicki, A common fixed point
theorem for a commuting family of weak$^{\ast }$ continuous nonexpansive
mappings, Studia Math. 225 (2014), 173--181.

\bibitem{Bo} R. D. Bourgin, Geometric aspects of convex sets with the
Radon-Nikod\'{y}m property, Springer-Verlag, Berlin 1983.

\bibitem{Br1} R. E. Bruck, Jr., Properties of fixed-point sets of
nonexpansive mappings in Banach spaces, Trans. Amer. Math. Soc. 179 (1973),
251--262.

\bibitem{Br2} R. E. Bruck, Jr., A common fixed point theorem for a commuting
family of nonexpansive mappings, Pacific J. Math. 53 (1974), 59--71.

\bibitem{Da} M. M. Day, Fixed point theorem for compact convex sets,
Illinois J. Math. 5 (1961) 585--590.

\bibitem{De} R. DeMarr, Common fixed points for commuting contraction
mappings, Pacific J. Math. 13 (1963), 1139--1141.

\bibitem{GlMe} E. Glasner, M. Megrelishvili, On fixed point theorems and
nonsensitivity, Israel J. Math. 190 (2012), 289--305.

\bibitem{HWW} P. Harmand, D. Werner, W. Werner, M-Ideals in Banach Spaces
and Banach Algebras, Springer-Verlag, Berlin, 1993.

\bibitem{Hsu} R. Hsu, Topics on weakly almost periodic functions, PhD
thesis, SUNY at Buffalo, 1985.

\bibitem{JaRo} J. E. Jayne, C. A. Rogers, $K$-analytic sets, in Analytic
sets, Academic Press, London-New York, 1980, 1-181.

\bibitem{La} A. T.-M. Lau, Invariant means on almost periodic functions and
fixed point properties, Rocky Mountain J. Math. 3 (1973), 69--76.

\bibitem{La1} A. T.-M. Lau, Some fixed point theorems and W*-algebras, in:
Fixed Point Theory and Applications, S. Swaminathan (ed.), Academic Press,
New York, 1976, 121--129.

\bibitem{La2} A. T.-M. Lau, Amenability and fixed point property for
semigroup of nonexpansive mappings, in: Fixed Point Theory and Applications,
M.A. Thera, J.B. Baillon (eds.), Longman Sci. Tech., Harlow, 1991, 303--313.

\bibitem{LaTa} A. T.-M. Lau, W. Takahashi, Invariant means and fixed point
properties for non-expansive representations of topological semigroups,
Topol. Methods Nonlinear Anal. 5 (1995), 39--57.

\bibitem{LaZh1} A. T.-M. Lau, Y. Zhang, Fixed point properties of semigroups
of non-expansive mappings, J. Funct. Anal. 254 (2008), no. 10, 2534--2554.

\bibitem{LaZh2} A. T.-M. Lau, Y. Zhang, Fixed point properties for
semigroups of nonlinear mappings and amenability, J. Funct. Anal. 263
(2012), 2949--2977.

\bibitem{LaZh4} A. T.-M. Lau, Y. Zhang, Algebraic and analytic properties of
semigroups related to fixed point properties of non-expansive mappings,
Semigroup Forum (2019), https://doi.org/10.1007/s00233-019-10048-7.

\bibitem{Me} M. Megrelishvili, Fragmentability and continuity of semigroup
actions, Semigroup Forum 57 (1998), 101--126.

\bibitem{Mi} T. Mitchell, Topological semigroups and fixed points, Illinois
J. Math. 14 (1970), 630--641.

\bibitem{Na1} I. Namioka, Radon-Nikod\'{y}m compact spaces and
fragmentability, Mathematika 34 (1987), 258--281.

\bibitem{Na4} I. Namioka, Fragmentability in Banach spaces: interaction of
topologies, Rev. R. Acad. Cienc. Exactas F\'{\i}s. Nat. Ser. A Math. RACSAM
104 (2010) 283--308.

\bibitem{Pa} A. L. T. Paterson, Amenability, American Mathematical Society,
Providence, RI, 1988.

\bibitem{Ry0} C. Ryll-Nardzewski, Generalized random ergodic theorems and
weakly almost periodic functions, Bull. Acad. Polon. Sci. S\'{e}r. Sci.
Math. Astronom. Phys. 10 (1962), 271--275.

\bibitem{Wi1} A. Wi\'{s}nicki, Amenable semigroups of nonexpansive mappings
on weakly compact convex sets, J. Nonlinear Convex Anal. 17 (2016),
2119--2127.
\end{thebibliography}
\end{document}